\documentclass[11pt,reqno]{amsart}
\usepackage{a4wide}
\usepackage{amssymb}
\usepackage{amsthm}
\usepackage{amsmath}
\usepackage{mathtools}
\usepackage{amscd}
\usepackage{verbatim,url}
\usepackage{color,cancel}
\usepackage[mathscr]{eucal}
\usepackage{cases}
\numberwithin{equation}{section}
\allowdisplaybreaks

\newtheorem{theorem}{Theorem}
\newtheorem{lemma}[theorem]{Lemma}

\newtheorem{proposition}[theorem]{Proposition}

\theoremstyle{remark}

\numberwithin{theorem}{section} \numberwithin{equation}{section}
\numberwithin{figure}{section}

\begin{document}
\title[]{Bailey pairs and quantum $q$-series identites. I. The Classical Identities}

\author{Jehanne Dousse}
\address{Université de Genève, 7--9, rue Conseil Général, 1205 Genève, Switzerland}
\email{jehanne.dousse@unige.ch}

\author{Jeremy Lovejoy}
\address{CNRS, Universit{\'e} Paris Cit\'e, B\^atiment Sophie Germain, Case Courier 7014,
8 Place Aur\'elie Nemours, 75205 Paris Cedex 13, FRANCE}
\email{lovejoy@math.cnrs.fr}

\date{\today}

\dedicatory{For Mourad Ismail on the occasion of his 80th birthday}

\begin{abstract}
We use Bailey pairs to prove $q$-series identities at roots of unity due to Cohen and Bryson--Ono--Pitman--Rhoades.    The proofs use Bailey pairs with quadratic forms developed in the study of mock theta functions.   In addition to the standard Bailey lemma, we require some changes-of-base established by Bressoud--Ismail--Stanton.    We then embed the identities in infinite families using the Bailey chain.
\end{abstract}

\maketitle

\section{Introduction}
Recall the standard $q$-series notation,
\begin{equation*}
(a)_n = (a;q)_n = \prod_{k=0}^{n-1} (1-aq^k),
\end{equation*}
valid for non-negative integers $n$.   In a study of $q$-series from Ramanujan's lost notebook, Cohen \cite{Co1} proved that for any root of unity $q$,
\begin{equation} \label{Cohen}
\sum_{n \geq 0} (-1)^n(q^{-1};q^{-1})_n = \sum_{n \geq 0} (q^2;q^2)_{n}q^{n+1}.
\end{equation}
This is not an identity between complex functions in any classical sense.   But at any root of unity $\xi$ both series truncate, leaving a polynomial identity in $\xi$.   For example at $q=i$ the left-hand side is
\begin{equation*}
1 - (1+i)+ (1+i)(1-i^2) - (1+i)(1-i^2)(1+i^3) = -2+i,
\end{equation*}
while the right-hand side is
\begin{equation*}
i + i^2(1-i^2) = -2+i.
\end{equation*}
Cohen also left two similar identities as exercises for the reader, 
 
\begin{numcases} 
{\sum_{n \geq 0} (q^2;q^2)_nq^{n+1} =}
 - \sum_{n \geq 0} (-q^{-1};q^{-1})_n, & \text{if $q$ is an even root of unity}, \label{Cohenex1} \\
\sum_{n \geq 0} (q;q^2)_n, & \text{if $q$ is an odd root of unity}. \label{Cohenex2}
\end{numcases}
Later, in a study of quantum modular forms, Bryson, Ono, Pitman, and Rhoades \cite{Br-On-Pi-Rh1} proved another elegant identity at roots of unity,
\begin{equation} \label{BOPR}
\sum_{n \geq 0} (q^{-1};q^{-1})_n = \sum_{n \geq 0} (q)_n^2q^{n+1}.
\end{equation} 


These early examples of Cohen and Bryson--Ono--Pitman--Rhoades laid the foundation for future work on $q$-series identites at roots of unity, also known as \emph{quantum $q$-series identities}.    While the original proofs  are based on elementary recursions, the identities can also be deduced using classical $q$-series transformations, as can a great many other such identities \cite{Lo4}.    The identity \eqref{BOPR} can also be deduced using the colored Jones polynomial of the trefoil knot \cite{Hi-Lo1}, and this observation led to many further families of identities at roots of unity, including several generalizations of \eqref{BOPR} \cite{Hi-Lo1, Lo-Os1, Lo-Os2}.   Quantum $q$-series identities also arise as a consequence of so-called ``strange identities" for quantum modular forms \cite{Lo3}.

In this paper we show how Bailey pairs may be used to prove \eqref{Cohen} -- \eqref{BOPR}.    Our proofs use Bailey pairs with quadratic forms 
developed
in the study of mock theta functions.    We use these pairs to express the relevant $q$-series as truncated indefinite theta series.   See Propositions \ref{BOPRprop} -- \ref{Cohenex2prop}.    Surprisingly, for two of the four identities a key role is played by changes-of-base for Bailey pairs due to Bressoud--Ismail--Stanton \cite{Br-Is-St1}.

Once we have understood the identities using Bailey pairs, they can be embedded in infinite families using the machinery of the Bailey chain.    For example, we will show the following.
\begin{theorem} \label{theorem1}
For $m \geq 1$ and $q$ any root of unity we have
\begin{equation} \label{theorem1eq}
\begin{aligned}
\sum_{n_m \geq \dots \geq n_1 \geq 0}& (q)_{n_m}(-1)^{n_1}q^{-\binom{n_1+1}{2}}\prod_{i=1}^{m-1} q^{n_i^2+n_i} \begin{bmatrix} n_{i+1} \\ n_i \end{bmatrix} \\
&=  \sum_{n_m \geq \dots \geq n_1 \geq 0} (q)_{n_m}(q)_{n_1}q^{n_m+1}\prod_{i=1}^{m-1} q^{n_i^2+n_i} \begin{bmatrix} n_{i+1} \\ n_i \end{bmatrix}.
\end{aligned}
\end{equation}
\end{theorem}   

\begin{theorem} \label{theorem2}
For $m \geq 1$ and $q$ any even root of unity we have
\begin{equation} \label{theorem2eq}
\begin{aligned}
\sum_{n_m \geq \dots \geq n_1 \geq 0}& (-q)_{n_m}(q)_{n_1}q^{n_m+1}\prod_{i=1}^{m-1} q^{n_i^2+n_i} \begin{bmatrix} n_{i+1} \\ n_i \end{bmatrix} \\
&=  -\sum_{n_m \geq \dots \geq n_1 \geq 0} (-q)_{n_m}(-1)^{n_m+ n_1}q^{- \binom{n_1+1}{2}}\prod_{i=1}^{m-1} q^{n_i^2+n_i} \begin{bmatrix} n_{i+1} \\ n_i \end{bmatrix}.
\end{aligned}
\end{equation}
\end{theorem} 
\noindent Here we have used the $q$-binomial coefficient (or Gaussian polynomial),
\begin{equation*} \label{qbinomial}
\begin{bmatrix} n \\ k \end{bmatrix} = \begin{bmatrix} n \\ k \end{bmatrix}_q =
\begin{cases}
\frac{(q)_n}{(q)_{n-k}(q)_k}, & \text{if $0 \leq k \leq n$}, \\
0, & \text{otherwise}.
\end{cases}
\end{equation*}
Note that the cases $m=1$ of \eqref{theorem1eq} and \eqref{theorem2eq} are \eqref{BOPR} and \eqref{Cohenex1}, respectively.   
This follows from the relation
\begin{equation} \label{inverseqfac}
(q^{-1};q^{-1})_n = (q)_n(-1)^nq^{-\binom{n+1}{2}}.
\end{equation}

Generalizations of \eqref{Cohen} and \eqref{Cohenex2} are contained in the following two results.   Here the $q$-series don't always truncate naturally, so the multisums come with an upper bound.

\begin{theorem} \label{theorem3}
For $m \geq 1$ and $q$ any primitive $N$th root of unity we have
\begin{equation*} \label{theorem3eq}
\begin{aligned}
&\sum_{N-1 \geq n_{2m-1} \geq \cdots  \geq n_1 \geq 0} (-q^{n_1+1})_{n_{2m-1}-n_1} (-1)^{n_{2m-1}} (q)_{n_1} \prod_{i=1}^{2m-2} q^{n_i^2+n_i} \begin{bmatrix} n_{i+1} \\ n_i \end{bmatrix}  \\
&= \sum_{N-1 \geq n_m \geq \cdots \geq n_1 \geq 0} (-q^{2n_1+2})_{2n_m-2n_1}(-1)^{n_m}q^{-(n_m+1)^2}(q^2;q^2)_{n_1} \prod_{i=1}^{m-1} q^{2n_i^2+2n_i} \begin{bmatrix} n_{i+1} \\ n_i \end{bmatrix}_{q^2}.
\end{aligned}
\end{equation*}
\end{theorem}

\begin{theorem} \label{theorem4}
For $m \geq 1$ and $q$ any primitive odd $N$th root of unity we have
\begin{equation*} \label{theorem4eq}
\begin{aligned}
&\sum_{N-1 \geq n_m \geq \cdots \geq n_1 \geq 0} (-q^{2n_1+2})_{2n_m-2n_1}(-1)^{n_m}q^{-(n_m+1)^2}(q^2;q^2)_{n_1} \prod_{i=1}^{m-1} q^{2n_i^2+2n_i} \begin{bmatrix} n_{i+1} \\ n_i \end{bmatrix}_{q^2} \\
&= \sum_{N-1 \geq n_m \geq \cdots \geq n_1 \geq 0} (-q^{2n_1+1})_{2n_m-2n_1}(-1)^{n_m}q^{-n_m^2}(q;q^2)_{n_1} \prod_{i=1}^{m-1} q^{2n_i^2+2n_i} \begin{bmatrix} n_{i+1} \\ n_i \end{bmatrix}_{q^2}.
\end{aligned}
\end{equation*} 
\end{theorem}
\noindent Note that when $m=1$ and $q=1/q$ Theorems \ref{theorem3} and \ref{theorem4} reduce to \eqref{Cohen} and \eqref{Cohenex2}.  

The rest of the paper is organized as follows.   In the next section we review the classical Bailey lemma along with some changes-of-base.   In Section 3 we prove the classical quantum $q$-series identities in \eqref{Cohen} -- \eqref{BOPR}.   In Section 4 we prove Theorems \ref{theorem1} -- \ref{theorem4}.   Future papers in this series will be devoted to a more thorough investigation of the role Bailey pairs play in proving quantum $q$-series identities as well as the role such identities play in establishing the quantum modularity of the relevant series.

\section{The Bailey lemma}     
We begin by reviewing some basic facts about Bailey pairs.   For more background, see \cite{An.5, An2,Wa1}.  A pair of sequences $(\alpha_n,\beta_n)$ is a Bailey pair relative to $a$ if
\begin{align} 
\beta_n &= \sum_{k=0}^n \frac{\alpha_k}{(q)_{n-k}(aq)_{n+k}} \label{pairdeforiginal} \\
& = \frac{1}{(q)_n(aq)_n} \sum_{k=0}^n \frac{(q^{-n})_k}{(aq^{n+1})_k}(-1)^kq^{nk -\binom{k}{2}} \alpha_k.   \label{pairdef}
\end{align}
Equation \eqref{pairdeforiginal} is the original definition, while \eqref{pairdef} follows using the relation
\begin{equation} \label{qPochn-k}
(x)_{n-k} = \frac{(x)_n}{(q^{1-n}/x)_k}(-q/x)^kq^{\binom{k}{2} - nk}.
\end{equation}
The following, known as the Bailey lemma and due to Andrews \cite{An.5}, produces new Bailey pairs from a given pair.   
\begin{lemma} \label{Baileylemma}
If $(\alpha_n,\beta_n)$ is a Bailey pair relative to $a$, then so is $(\alpha_n',\beta_n')$, where
\begin{equation} \label{alphaprime}
\alpha_n' = \frac{(b)_n(c)_n (aq/bc)^n}{(aq/b)_n(aq/c)_n} \alpha_n
\end{equation}
and
\begin{align} 
\beta_n' &=  \frac{1}{(aq/b,aq/c)_n} \sum_{k=0}^n \frac{(b)_k(c)_k(aq/bc)_{n-k} (aq/bc)^k}{(q)_{n-k}} \beta_k \label{betaprime} \\
&= \frac{(aq/bc)_n}{(q)_n(aq/b)_n(aq/c)_n} \sum_{k=0}^n \frac{(b)_k(c)_k(q^{-n})_kq^k}{(bcq^{-n}/a)_k}\beta_k \label{betaprimebis}.
\end{align}
\end{lemma}
Repeated application of the Bailey lemma is called the Bailey chain.   We record the cases $b,c \to \infty$ and $b,c \to 0$ of \eqref{betaprime} with $a=q$.
\begin{lemma} \label{Baileylemmabcinfty}
If $(\alpha_n,\beta_n)$ is a Bailey pair relative to $q$, then so is $(\alpha_n', \beta_n')$, where
\begin{equation*}
\alpha_n'  = q^{n^2+n}\alpha_n
\end{equation*}
and
\begin{equation*}
\beta_n' = \sum_{k=0}^n \frac{q^{k^2+k}}{(q)_{n-k}}\beta_k.
\end{equation*}
\end{lemma}

\begin{lemma} \label{Baileylemmabc0}
If $(\alpha_n,\beta_n)$ is a Bailey pair relative to $q$, then so is $(\alpha_n', \beta_n')$, where
\begin{equation*}
\alpha_n'  = q^{-n^2-n}\alpha_n
\end{equation*}
and
\begin{equation*}
\beta_n' = (-1)^nq^{-\binom{n+1}{2}-n}\sum_{k=0}^n \frac{q^{\binom{k+1}{2}-nk}(-1)^k}{(q)_{n-k}}\beta_k.
\end{equation*}
\end{lemma}

Using \eqref{alphaprime} and \eqref{betaprimebis} in \eqref{pairdef} with $n=n-1$ and $a=q$ gives a key identity, which we state as a lemma.
\begin{lemma} \label{keylemma1}
If $(\alpha_n,\beta_n)$ is a Bailey pair relative to $q$, then we have
\begin{equation*} \label{keyrational}
\frac{(q^2/bc)_{n-1}(q^2)_{n-1}}{(q^2/b)_{n-1}(q^2/c)_{n-1}}\sum_{k=0}^{n-1} \frac{(b)_k(c)_k(q^{1-n})_kq^k}{(bcq^{-n})_k}\beta_k  =
\sum_{k=0}^{n-1} \frac{(q^{1-n})_k(b)_k(c)_kq^{nk-\binom{k+1}{2}} (-q^2/bc)^k}{(q^{1+n})_k(q^2/b)_k(q^2/c)_k}\alpha_k.
\end{equation*}
\end{lemma}

Next, we record two changes-of-base for Bailey pairs due to Bressoud, Ismail, and Stanton \cite{Br-Is-St1}.   The first is the case $a=q$ of \cite[D(1)]{Br-Is-St1} while the second is the case $a=q$ of \cite[D(4)]{Br-Is-St1}.
\begin{lemma} \label{basechangelemma1}
If $(\alpha_n,\beta_n)$ is a Bailey pair relative to $q$, then so is $(\alpha_n', \beta_n')$, where 
\begin{equation*}
\alpha_n' = \alpha_n(q^2),
\end{equation*}
and
\begin{equation*}
\beta_n'= \sum_{k=0}^n \frac{(-q^2)_{2k}}{(q^2;q^2)_{n-k}}q^{n-k}\beta_k(q^2),
\end{equation*}
\end{lemma}

\begin{lemma} \label{basechangelemma2}
If $(\alpha_n,\beta_n)$ is a Bailey pair relative to $q$, then so is $(\alpha_n', \beta_n')$, where 
\begin{equation*}
\alpha_n' = \frac{1+q}{1+q^{2n+1}}q^n\alpha_n(q^2),
\end{equation*}
and
\begin{equation*}
\beta_n' = \sum_{k=0}^n \frac{(-q)_{2k}}{(q^2;q^2)_{n-k}}q^k\beta_k(q^2).
\end{equation*}
\end{lemma}

Using these in \eqref{pairdef} and rewriting using \eqref{qPochn-k} gives two more key identities.   
\begin{lemma} \label{keylemma2}
If $(\alpha_n,\beta_n)$ is a Bailey pair relative to $q$, then 
\begin{equation*}
\begin{aligned}
\sum_{k=0}^{n-1}(-q^2)_{2k}&(q^{2-2n};q^2)_k(-1)^kq^{n(2k+1)-k^2-2k-1} \beta_k(q^2) \\
&= \frac{(-q)_{n-1}}{(q^2)_{n-1}}\sum_{k=0}^{n-1} \frac{(q^{1-n})_k}{(q^{1+n})_k}(-1)^kq^{nk-\binom{k+1}{2}} \alpha_k(q^2).
\end{aligned}
\end{equation*}
\end{lemma}

\begin{lemma} \label{keylemma3}
If $(\alpha_n,\beta_n)$ is a Bailey pair relative to $q$, then 
\begin{equation*}
\begin{aligned}
\sum_{k=0}^{n-1}(-q)_{2k}&(q^{2-2n};q^2)_k(-1)^kq^{2nk-k^2} \beta_k(q^2) \\
&= \frac{(-q)_{n-1}}{(q^2)_{n-1}}\sum_{k=0}^{n-1} \frac{(q^{1-n})_k(1+q)}{(q^{1+n})_k(1+q^{2k+1})}(-1)^kq^{nk-\binom{k}{2}} \alpha_k(q^2).
\end{aligned}
\end{equation*}
\end{lemma}

Finally, we note five Bailey pairs. 

\begin{lemma} \label{fivepairslemma}
The following are Bailey pairs relative to $q$.
\begin{align}
\alpha_k &= \frac{(1-q^{2k+1})q^{-k}}{1-q}\sum_{j=-k}^k(-1)^jq^{j(3j+1)/2} \ \ \ \text{and} \ \ \  \beta_k = \frac{q^{-k}}{(q)_k}, \label{posdefpair2} \\
\alpha_k &= \frac{(1-q^{2k+1})}{1-q}q^{2k^2+k}\sum_{j=-k}^k (-1)^jq^{-j(3j+1)/2} \ \ \ \text{and } \ \ \ \beta_k = 1, \label{indefpair1} \\
\alpha_k &= \frac{(1-q^{2k+1})}{1-q}q^{k(3k+1)/2}\sum_{j = -k}^k (-1)^jq^{-j^2} \ \ \ \text{and} \ \ \ \beta_k = \frac{1}{(-q)_k}, \label{indefpair2} \\
\alpha_k &= \frac{(1-q^{k+1/2})}{1-q^{1/2}}q^{k^2+k/2}\sum_{j=-k}^k (-1)^jq^{-j^2/2} \ \ \ \text{and} \ \ \ \beta_k = \frac{1}{(-q;q^{1/2})_{2k}}, \label{indefpair4}\\
\alpha_k &= \frac{(1-q^{2k+1})}{(1-q)} q^{k^2} \sum_{j=-k}^k (-1)^j q^{-j^2/2} \ \ \ \text{and} \ \ \ 
\beta_k = \frac{(q^{1/2};q)_k}{(q)_k(-q^{1/2};q^{1/2})_{2k}}.  \label{eq:BP1.4}
\end{align}
\end{lemma}

\begin{proof}
The Bailey pairs \eqref{indefpair1} and \eqref{indefpair2} are due to Andrews \cite{An1}, while the pair \eqref{indefpair4} arose in work of Andrews and Hickerson \cite{An-Hi1}.     The remaining pairs are easily deduced from \cite[Theorem 7]{Lo1}.   Namely, \eqref{posdefpair2} is the case $b,c,d \to \infty$ and \eqref{eq:BP1.4} is the case $b=-1$, $c=-q^{1/2}$, and $d=0$.
\end{proof}

\section{Proofs of \eqref{Cohen} -- \eqref{BOPR}}
In this section we prove the quantum $q$-series identities in \eqref{Cohen} -- \eqref{BOPR}.    In each case we use Bailey pairs to express both sides of the identity as truncated indefinite theta series.   We begin with \eqref{BOPR}, which will follow immediately from the next proposition.    We use the standard convention that for any $f$,
\begin{equation*} \label{convention}
\sum_{j=a}^b f(j) = - \sum_{j=b+1}^{a-1} f(j)
\end{equation*}
whenever $a > b$.  
\begin{proposition} \label{BOPRprop}
For $q$ a primitive $n$th root of unity we have
\begin{align}
\sum_{k=0}^{n-1} (q)_k &= \frac{-1}{n} \sum_{k=-n}^{n-1} \sum_{j=-k}^k (k^2-j(3j+1)/2) q^{-k^2 + j(3j+1)/2}(-1)^j, \label{BOPRpropeq1}\\
\sum_{k=0}^{n-1} (q)_k^2q^{k+1} &= \frac{-1}{n} \sum_{k=-n}^{n-1} \sum_{j=-k}^k (k^2-j(3j+1)/2)q^{k^2-j(3j+1)/2)}(-1)^j.  \label{BOPRpropeq2}
\end{align}
\end{proposition}
 
\begin{proof}
We begin with \eqref{BOPRpropeq1}.   Letting $b \to 0$ and setting $c=q$ in Lemma \ref{keylemma1} we have that if $(\alpha_k,\beta_k)$ is a Bailey pair relative to $q$, then  
\begin{equation} \label{b0cq}
\frac{1-q^n}{1-q}\sum_{k =0}^{n-1} (q)_k(q^{1-n})_kq^{k+1} \beta_k = q^n\sum_{k=0}^{n-1} \frac{(q^{1-n})_k}{(q^{1+n})_k}q^{nk - (k^2+k)}\alpha_k.
\end{equation}
Using the Bailey pair \eqref{posdefpair2} in \eqref{b0cq} we have the rational function identity
\begin{equation*}
(1-q^n)\sum_{k =0}^{n-1} (q^{1-n})_k = q^n \sum_{k=0}^{n-1} \frac{(q^{1-n})_k}{(q^{1+n})_k}q^{nk-k^2-2k-1}(1-q^{2k+1})\sum_{j=-k}^k(-1)^jq^{j(3j+1)/2}.
\end{equation*}
In light of the factor $(1-q^n)$ on the left-hand side, we see that the right-hand side vanishes when $q$ is an $n$th root of unity.     Using l'H\^opital's rule we obtain that if $q = \zeta_n$ is a primitive $n$th root of unity, then  
\begin{align*}
\sum_{k=0}^{n-1} (q)_k &=  \lim_{q \to \zeta_n} \frac{1}{1-q^n} \sum_{k=0}^{n-1} q^{-k^2-2k-1}(1-q^{2k+1})\sum_{j=-k}^k (-1)^jq^{j(3j+1)/2}\\
&= \frac{-q}{n} \frac{d}{dq}\Big|_{q= \zeta_n} \sum_{k=0}^{n-1} q^{-k^2-2k-1}(1-q^{2k+1})\sum_{j=-k}^k (-1)^jq^{j(3j+1)/2} \\
&= \frac{q}{n} \frac{d}{dq}\Big|_{q= \zeta_n} \sum_{k=-n}^{n-1}\sum_{j=-k}^k q^{-k^2+j(3j+1)/2}(-1)^j,
\end{align*}
and \eqref{BOPRpropeq1} follows.   In the above we have used the fact that for $q$ a primitive $n$th root of unity and $0 \leq k \leq n-1$, $$\frac{(q^{1-n})_k}{(q^{1+n})_k} =1.$$    For \eqref{BOPRpropeq2}, we use \eqref{indefpair1} in \eqref{b0cq} and perform a similar calculation.   This completes the proof.
\end{proof}

Next we treat \eqref{Cohenex1}.    This follows immediately from the next proposition.    We use the sign function 
\begin{equation*}
{\rm sgn}(k) = 
\begin{cases}
1, & \text{if $k \geq 0$}, \\
-1, & \text{if $k < 0$}.
\end{cases}
\end{equation*} 
\begin{proposition} \label{Cohenex1prop}
If $q$ is a primitive $n$th root of unity we have
\begin{equation} 
\sum_{k=0}^{n-1} (q^2;q^2)_kq^{k+1} = 
\begin{dcases}
\frac{-1}{n^2} \sum_{k=-n}^{n-1} \sum_{j=-k}^k {\rm sgn}(k) (k^2-j(3j+1)/2)(-1)^{k+j}q^{k^2-j(3j+1)/2}, \ \text{$n$ odd}, \\
\frac{-1}{4} \sum_{k=-n}^{n-1} \sum_{j=-k}^k {\rm sgn}(k)(-1)^{k+j}q^{k^2-j(3j+1)/2},  \ \text{$n$ even},
\end{dcases} \label{Cohenex1propeq1}
\end{equation}
and if $q$ is a primitive even $n$th root of unity then 
\begin{equation}
\sum_{k=0}^{n-1} (-q)_k =  \frac{1}{4} \sum_{k=-n}^{n-1} \sum_{j=-k}^k {\rm sgn}(k)(-1)^{k+j}q^{-k^2+j(3j+1)/2}.     \label{Cohenex1propeq2}
\end{equation}
\end{proposition}

\begin{proof}
If we take $b \to 0$ and set $c=-q$ in Lemma \ref{keylemma1} we have that if $(\alpha_k,\beta_k)$ is a Bailey pair relative to $q$, then
\begin{equation} \label{b0c-q}
\sum_{k =0}^{n-1} (-q)_k(q^{1-n})_kq^{k+1} \beta_k = -(-q)^{n}\frac{(-q)_{n-1}}{(q^2)_{n-1}}\sum_{k=0}^{n-1} \frac{(q^{1-n})_k}{(q^{1+n})_k}q^{nk - (k^2+k)}(-1)^k\alpha_k.
\end{equation}
Using the Bailey pair \eqref{indefpair1} gives the rational function identity
\begin{equation*}
\begin{aligned}
\sum_{k =0}^{n-1}& (-q)_k(q^{1-n})_kq^{k+1} \\ 
&= -(-q)^{n}\frac{(-q)_{n-1}}{(q)_{n}}\sum_{k=0}^{n-1} \frac{(q^{1-n})_k}{(q^{1+n})_k}q^{nk +k^2}(-1)^k (1-q^{2k+1})\sum_{j=-k}^k (-1)^jq^{-j(3j+1)/2}.
\end{aligned}
\end{equation*}
We first consider the case when $q$ is a primitive odd $n$th root of unity $\zeta_n$.    In this case, we have $(-q)_{n-1} = 1$ and $(q)_{n-1}=n$.   Hence we have
\begin{align*}
\sum_{k=0}^{n-1} (q^2;q^2)_{k}q^{k+1} &= \lim_{q \to \zeta_n} \frac{1}{n(1-q^n)} \sum_{k=0}^{n-1}\sum_{j=-k}^k (-1)^{j+k}q^{k^2-j(3j+1)/2}(1-q^{2k+1}) \\
&= \frac{-q}{n^2} \frac{d}{dq}\Big|_{q = \zeta_n} \sum_{k=-n}^{n-1}\sum_{j=-k}^k {\rm sgn}(k)(-1)^{j+k}q^{k^2-j(3j+1)/2},
\end{align*}
Equation \eqref{Cohenex1propeq1} follows for $n$ odd.

Next assume that $n$ is even.  In this case, we still have $(q)_{n-1} = n$ but now $(-q)_{n-1}=0$.   A short calculation gives that for a primitive even $n$th root of unity $\zeta_n$, we have $\frac{d}{dq}\big|_{q=\zeta_n} (-q)_{n-1} = n^2/4\zeta_n$.   With this in mind, for $q = \zeta_n$ we obtain
\begin{align*}
\sum_{k=0}^{n-1} (q^2;q^2)_{k}q^{k+1} &= \lim_{q \to \zeta_n} \frac{-(-q)_{n-1}}{n(1-q^n)} \sum_{k=0}^{n-1}\sum_{j=-k}^k (-1)^{j+k}q^{k^2-j(3j+1)/2}(1-q^{2k+1}) \\
&= \frac{1}{4} \sum_{k=-n}^{n-1}\sum_{j=-k}^k {\rm sgn}(k)(-1)^{j+k}q^{k^2-j(3j+1)/2}.
\end{align*}

This completes the proof of \eqref{Cohenex1propeq1}.   Equation \eqref{Cohenex1propeq2} follows in a similar manner using the Bailey pair \eqref{posdefpair2} in \eqref{b0c-q}.    In obtaining the left-hand side of \eqref{Cohenex1propeq2} from \eqref{b0c-q} we use the fact that if $q$ is a primitive $n$th root of unity, then for $0 \leq k \leq n-1$ we have
\begin{equation*}
\frac{(q^{1-n})_k}{(q)_k} =1.
\end{equation*}
\end{proof}

We now move on to \eqref{Cohen}, which is an immedaite consequence of the following proposition.
\begin{proposition} \label{Cohenprop}
For $q$ a primitive $n$th root of unity we have
\begin{equation} \label{Cohenpropeq1}
\sum_{k=0}^{n-1} (q)_k(-1)^k = 
\begin{dcases}
\frac{-1}{n^2} \sum_{k=-n}^{n-1} \sum_{j=-k}^k {\rm sgn}(k) (k(3k+1)/2 - j^2)(-1)^{k+j}q^{k(3k+1)/2 - j^2}, \ \text{$n$ odd}, \\
\frac{1}{4} \sum_{k=-n}^{n-1} \sum_{j=-k}^k {\rm sgn}(k)(-1)^{k+j}q^{k(3k+1)/2 - j^2},  \ \text{$n$ even},
\end{dcases}
\end{equation}
and
\begin{equation} \label{Cohenpropeq2}
\begin{aligned}
\sum_{k=0}^{n-1}& (q^2;q^2)_kq^{k+1} \\
&= 
\begin{dcases}
\frac{-1}{n^2} \sum_{k=-n}^{n-1} \sum_{j=-k}^k {\rm sgn}(k) (k(3k+1)/2 - j^2)(-1)^{k+j}q^{-k(3k+1)/2 + j^2}, \ \text{$n$ odd}, \\
\frac{1}{4} \sum_{k=-n}^{n-1} \sum_{j=-k}^k {\rm sgn}(k)(-1)^{k+j}q^{-k(3k+1)/2 + j^2},  \ \text{$n$ even}.
\end{dcases}
\end{aligned}
\end{equation}
\end{proposition}

\begin{proof}
Applying Lemma \ref{keylemma1} with $b \to \infty$ and $c=-q$,   we find that if $(\alpha_n,\beta_n)$ is a Bailey pair relative to $q$, then
\begin{equation} \label{binftyc-q}
\sum_{k =0}^{n-1} (-q)_k(q^{1-n})_k(-1)^kq^{nk} \beta_k = \frac{(-q)_{n-1}}{(q^2)_{n-1}}\sum_{k=0}^{n-1} \frac{(q^{1-n})_k}{(q^{1+n})_k}q^{nk}(-1)^k\alpha_k.
\end{equation}
If we use the Bailey pair \eqref{indefpair2} and argue as usual, we obtain \eqref{Cohenpropeq1}.

Now we take Lemma \ref{keylemma2} and insert the Bailey pair \eqref{indefpair4} to obtain the identity
\begin{equation*}
\begin{aligned}
\sum_{k=0}^{n-1}&(q^{2-2n};q^2)_k(-1)^kq^{n(2k+1)-k^2-2k-1}  \\
&= \frac{(-q)_{n-1}}{(q)_{n}}\sum_{k=0}^{n-1} \frac{(q^{1-n})_k}{(q^{1+n})_k}(-1)^kq^{nk+k(3k+1)/2} (1-q^{2k+1}) \sum_{j=-k}^k (-1)^jq^{-j^2}.
\end{aligned}
\end{equation*}
If $q$ is a primitive $n$th root of unity, this leads to
\begin{equation*} \label{CohenRHSinverse}
\begin{aligned}
\sum_{k=0}^{n-1}& (q^2;q^2)_k(-1)^kq^{-k^2-2k-1} \\
&= 
\begin{dcases}
\frac{-1}{n^2} \sum_{k=-n}^{n-1} \sum_{j=-k}^k {\rm sgn}(k) (k(3k+1)/2 - j^2)(-1)^{k+j}q^{k(3k+1)/2 - j^2}, \ \text{$n$ odd}, \\
\frac{1}{4} \sum_{k=-n}^{n-1} \sum_{j=-k}^k {\rm sgn}(k)(-1)^{k+j}q^{k(3k+1)/2 - j^2},  \ \text{$n$ even}.
\end{dcases}
\end{aligned}
\end{equation*}
Replacing $q$ by $q^{-1}$ gives \eqref{Cohenpropeq2}.
\end{proof}

We conclude with \eqref{Cohenex2}, which will follow from the proposition below combined with \eqref{Cohenpropeq2}.

\begin{proposition} \label{Cohenex2prop}
If $q$ is a primitive odd $n$th root of unity, we have
\begin{equation*} 
\label{Cohenex2propeq}
\sum_{n \geq 0} (q;q^2)_n = \frac{-1}{n^2} \sum_{k=-n}^{n-1} \sum_{j=-k}^k {\rm sgn}(k) (k(3k+1)/2 - j^2)(-1)^{k+j}q^{-k(3k+1)/2 + j^2}.
\end{equation*}
\end{proposition}
\begin{proof}
Inserting \eqref{eq:BP1.4} into Lemma \ref{keylemma3} gives that for $q$ any primitive $n$th root of unity,
\begin{align*}
\sum_{k=0}^{n-1} (q;q^2)_k(-1)^kq^{-k^2}
&= \frac{(-q)_{n-1}}{(q)_{n}} \sum_{k=0}^{n-1}(-1)^kq^{\frac{k(3k+1)}{2}} (1-q^{2k+1}) \sum_{j=-k}^k (-1)^j q^{-j^2}.
\end{align*}
When $q= \zeta_n$, a primitive odd $n$th root of unity, we obtain
\begin{align*}
\sum_{k=0}^{n-1} (q;q^2)_k(-1)^kq^{-k^2}
&= \lim_{q \to \zeta_n} \frac{1}{n(1-q^n)} \sum_{k=0}^{n-1}(-1)^kq^{\frac{k(3k+1)}{2}} (1-q^{2k+1}) \sum_{j=-k}^k (-1)^j q^{-j^2}\\
&=\frac{-q}{n^2} \frac{d}{dq}\Big|_{q = \zeta_n} \sum_{k=0}^{n-1}(-1)^kq^{\frac{k(3k+1)}{2}} (1-q^{2k+1}) \sum_{j=-k}^k (-1)^j q^{-j^2}\\
&=\frac{-1}{n^2} \sum_{k=-n}^{n-1} \sum_{j=-k}^k {\rm sgn}(k) (k(3k+1)/2 - j^2)(-1)^{k+j}q^{k(3k+1)/2 - j^2}.
\end{align*}
Replacing $q$ by $q^{-1}$ gives the desired result.
\end{proof}

\section{Proofs of Theorems \ref{theorem1} -- \ref{theorem4}}
In this section we use the Bailey chain to prove the generalizations of the identities of Cohen and Bryson--Ono--Pitman--Rhoades contained in Theorems \ref{theorem1} -- \ref{theorem4}.        

\begin{proof}[Proof of Theorem \ref{theorem1}]
 We begin by applying Lemma \ref{Baileylemmabcinfty} $m-1$ times to the Bailey pair in \eqref{indefpair1} for $m \geq 1$.   The result is the Bailey pair relative to $q$,
\begin{equation}
\label{eq:alpha1}
\alpha_n = \frac{q^{(m+1)n^2+mn}(1-q^{2n+1})}{1-q} \sum_{j=-n}^n (-1)^jq^{-j(3j+1)/2}
\end{equation}    
and
\begin{align}
\beta_n =\beta_{n_{m}} &= \sum_{n_{m} \geq \cdots \geq n_1 \geq 0} \frac{q^{n_{m-1}^2 + n_{m-1} + \cdots + n_1^2+n_1}}{(q)_{n_{m} - n_{m-1}}\cdots(q)_{n_2-n_1}} \nonumber \\
&= \frac{1}{(q)_{n_{m}}}\sum_{n_m \geq \cdots \geq n_1 \geq 0} (q)_{n_1} \prod_{i=1}^{m-1} q^{n_i^2+n_i} \begin{bmatrix} n_{i+1} \\ n_i \end{bmatrix}. \label{eq:beta1}
\end{align}  
Inserting this into \eqref{b0cq} and arguing as in the proof of \eqref{BOPR}, we find that for a primitive $n$th root of unity $q$,
\begin{equation} \label{theorem1halfway}
\begin{aligned}
&\sum_{n-1 \geq n_m \geq \dots \geq n_1 \geq 0} (q)_{n_m}(q)_{n_1}q^{n_m+1}\prod_{i=1}^{m-1} q^{n_i^2+n_i} \begin{bmatrix} n_{i+1} \\ n_i \end{bmatrix} \\
&= \frac{-1}{n} \sum_{k=-n}^{n-1}\sum_{j=-k}^k \left(mk^2 + (m-1)k -j(3j+1)/2\right)q^{mk^2 + (m-1)k -j(3j+1)/2}(-1)^j.
\end{aligned}
\end{equation}

Next we apply Lemma \ref{Baileylemmabc0} $m-1$ times to the Bailey pair in \eqref{posdefpair2}.   The result is the Bailey pair relative to $q$,
\begin{equation}
\label{eq:alpha2}
\alpha_n = \frac{q^{-(m-1)n^2-mn}(1-q^{2n+1})}{1-q}  \sum_{j=-n}^n (-1)^jq^{j(3j+1)/2}
\end{equation}
and
\begin{align}
\beta_n =\beta_{n_{m}} &= (-1)^{n_m}q^{-\binom{n_m+1}{2} - n_m}\sum_{n_{m} \geq \cdots \geq n_1 \geq 0} \frac{q^{-n_mn_{m-1}  - n_{m-1} - \cdots - n_2n_1 - n_1}(-1)^{n_1}q^{\binom{n_1+1}{2}}}{(q)_{n_{m} - n_{m-1}}\cdots(q)_{n_2-n_1}(q)_{n_1}} \nonumber \\
&= \frac{(-1)^{n_m}q^{-\binom{n_m+1}{2} - n_m}}{(q)_{n_{m}}}\sum_{n_m \geq \cdots \geq n_1 \geq 0} (-1)^{n_1}q^{\binom{n_1+1}{2}} \prod_{i=1}^{m-1} q^{-n_in_{i+1} - n_i} \begin{bmatrix} n_{i+1} \\ n_i \end{bmatrix}. \label{eq:beta2}
\end{align} 
Inserting this in \eqref{b0cq} and arguing as usual gives that for $q$ a primitive $n$th root of unity,
\begin{equation*}
\begin{aligned}
&\sum_{n-1 \geq n_m \geq \cdots \geq n_1 \geq 0} (-1)^{n_m+n_1}q^{-\binom{n_m+1}{2}+\binom{n_1+1}{2}} (q)_{n_m}\prod_{i=1}^{m-1} q^{-n_in_{i+1} - n_i} \begin{bmatrix} n_{i+1} \\ n_i \end{bmatrix} \\
&= \frac{-1}{n} \sum_{k=-n}^{n-1}\sum_{j=-k}^k \left(mk^2 + (m-1)k -j(3j+1)/2\right)q^{-mk^2 - (m-1)k + j(3j+1)/2}(-1)^j.
\end{aligned}
\end{equation*}
Now, in the above we let $q=1/q$ and invoke \eqref{inverseqfac} along with the fact that 
\begin{equation*}
\begin{bmatrix} n \\ k \end{bmatrix}_{q^{-1}} =  \begin{bmatrix} n \\ k \end{bmatrix} q^{k^2-nk}.
\end{equation*}
We obtain that if $q$ is a primitive $n$th root of unity, then 
\begin{equation*}
\begin{aligned}
&\sum_{n-1 \geq n_m \geq \cdots \geq n_1 \geq 0} (-1)^{n_1}q^{-\binom{n_1+1}{2}} (q)_{n_m}\prod_{i=1}^{m-1} q^{n_i^2+n_i} \begin{bmatrix} n_{i+1} \\ n_i \end{bmatrix} \\
&= \frac{-1}{n} \sum_{k=-n}^{n-1}\sum_{j=-k}^k \left(mk^2 + (m-1)k -j(3j+1)/2\right)q^{(m+1)k^2 + mk - j(3j+1)/2}(-1)^j.
\end{aligned}
\end{equation*}
Comparing with \eqref{theorem1halfway} gives the result.
\end{proof}

We now turn to the proof of Theorem \ref{theorem2}, which follows a similar principle.

\begin{proof}[Proof of Theorem \ref{theorem2}]
We start with the same Bailey pair as in the proof of Theorem \ref{theorem1}, that is the $(\alpha_n,\beta_n)$ given in \eqref{eq:alpha1} and \eqref{eq:beta1}, but now we insert it into \eqref{b0c-q}. The result is
\begin{align*}
&\sum_{n_m =0}^{n-1} (-q)_{n_m}(q^{1-n})_{n_m} q^{n_m+1} \frac{1}{(q)_{n_{m}}}\sum_{n_m \geq \cdots \geq n_1 \geq 0} (q)_{n_1} \prod_{i=1}^{m-1} q^{n_i^2+n_i} \begin{bmatrix} n_{i+1} \\ n_i \end{bmatrix} \\
&= -(-q)^{n}\frac{(-q)_{n-1}}{(q)_{n}}\sum_{k=0}^{n-1} \frac{(q^{1-n})_k}{(q^{1+n})_k}(-1)^k q^{nk+mk^2+(m-1)k}(1-q^{2k+1}) \sum_{j=-k}^k (-1)^jq^{-j(3j+1)/2}.
\end{align*}
Arguing as usual, we get that for a primitive even $n$th root of unity $q$,
\begin{align}
&\sum_{n-1 \geq n_m \geq \dots \geq n_1 \geq 0} (-q)_{n_m} (q)_{n_1} q^{n_m+1}  \prod_{i=1}^{m-1} q^{n_i^2+n_i} \begin{bmatrix} n_{i+1} \\ n_i \end{bmatrix} \nonumber \\
&= \lim_{q \to \zeta_n} \frac{-(-q)_{n-1}}{n(1-q^n)} \sum_{k=0}^{n-1}\sum_{j=-k}^k (-1)^{j+k}q^{mk^2+(m-1)k-j(3j+1)/2}(1-q^{2k+1}) \nonumber \\
&= \frac{1}{4} \sum_{k=-n}^{n-1}\sum_{j=-k}^k {\rm sgn}(k)(-1)^{j+k}q^{mk^2+(m-1)k-j(3j+1)/2}. \label{theorem2halfway}
\end{align}

Now we take the Bailey pair of \eqref{eq:alpha2} and \eqref{eq:beta2} and insert it into \eqref{b0c-q}. This gives
\begin{align*}
&\sum_{n-1 \geq n_m \geq \cdots \geq n_1 \geq 0}  (-q)_{n_m} (-1)^{n_1+n_m} (q^{1-n})_{n_m} \frac{q^{1-\binom{n_m+1}{2}+\binom{n_1+1}{2}}}{(q)_{n_{m}}}  \prod_{i=1}^{m-1} q^{-n_in_{i+1} - n_i} \begin{bmatrix} n_{i+1} \\ n_i \end{bmatrix} \\
&= -(-q)^{n}\frac{(-q)_{n-1}}{(q)_{n}}\sum_{k=0}^{n-1} \frac{(q^{1-n})_k}{(q^{1+n})_k}(-1)^kq^{nk - m k^2 -(m+1)k}(1-q^{2k+1}) \sum_{j=-k}^k (-1)^jq^{j(3j+1)/2}.
\end{align*}
Dividing both sides by $q$ and rearranging, we obtain that for a primitive even $n$th root of unity $q$,
\begin{align*}
&\sum_{n-1 \geq n_m \geq \cdots \geq n_1 \geq 0}  (-q)_{n_m} (-1)^{n_1+n_m} q^{-\binom{n_m+1}{2}+\binom{n_1+1}{2}} \prod_{i=1}^{m-1} q^{-n_in_{i+1} - n_i} \begin{bmatrix} n_{i+1} \\ n_i \end{bmatrix} \\
&= \lim_{q \to \zeta_n} \frac{(-q)_{n-1}}{n(1-q^n)} \sum_{k=0}^{n-1} \sum_{j=-k}^k (-1)^kq^{- m k^2 -(m-1)k}(1-q^{-(2k+1)})  (-1)^jq^{j(3j+1)/2} \nonumber \\
&= -\frac{1}{4} \sum_{k=-n}^{n-1}\sum_{j=-k}^k {\rm sgn}(k)(-1)^{j+k}q^{-mk^2-(m-1)k+j(3j+1)/2}.
\end{align*}

Letting $q=1/q$ in the above gives
\begin{equation*}
\begin{aligned}
&\sum_{n-1 \geq n_m \geq \cdots \geq n_1 \geq 0} (-q)_{n_m} (-1)^{n_1+n_m} q^{-\binom{n_1+1}{2}} \prod_{i=1}^{m-1} q^{n_i^2+n_i} \begin{bmatrix} n_{i+1} \\ n_i \end{bmatrix} \\
&= -\frac{1}{4} \sum_{k=-n}^{n-1}\sum_{j=-k}^k {\rm sgn}(k)(-1)^{j+k}q^{mk^2+(m-1)k-j(3j+1)/2}.
\end{aligned}
\end{equation*}
Comparing with \eqref{theorem2halfway} now gives the result.
\end{proof}

\begin{proof}[Proof of Theorem \ref{theorem3}]
We begin by applying Lemma \ref{Baileylemmabcinfty} $m-1$ times to the Bailey pair in \eqref{indefpair2} for $m \geq 1$.    The result is the Bailey pair relative to $q$,
\begin{equation*}
\alpha_n = \frac{q^{n(3n+1)/2 + (m-1)(n^2+n)}(1-q^{2n+1})}{1-q}\sum_{j=-n}^n(-1)^jq^{-j^2}
\end{equation*}
and
\begin{equation*}
\beta_n = \beta_{n_m} = \frac{1}{(q)_{n_m}}\sum_{n_m \geq \cdots \geq n_1 \geq 0} \frac{(q)_{n_1}}{(-q)_{n_1}}\prod_{i=1}^{m-1} q^{n_i^2+n_i} \begin{bmatrix} n_{i+1} \\ n_i \end{bmatrix}.
\end{equation*}
Inserting this into \eqref{binftyc-q}, letting $q$ be a primitive $n$th root of unity, and calculating as usual, we obtain
\begin{equation} \label{theorem3proofbigeq1}
\begin{aligned}
&\sum_{n-1 \geq n_m \geq \cdots \geq n_1 \geq 0} (-q^{n_1+1})_{n_m-n_1}(-1)^{n_m}(q)_{n_1}\prod_{i=1}^{m-1} q^{n_i^2+n_i} \begin{bmatrix} n_{i+1} \\ n_i \end{bmatrix} \\
&=
\begin{dcases}
\frac{-1}{n^2} \sum_{k=-n}^{n-1} \sum_{j=-k}^k {\rm sgn}(k) (k(3k+1)/2 +(m-1)(k^2+k) - j^2)\\
\qquad \qquad \qquad \times (-1)^{k+j}q^{k(3k+1)/2 + (m-1)(k^2+k)- j^2}, \qquad \text{$n$ odd}, \\
\frac{1}{4} \sum_{k=-n}^{n-1} \sum_{j=-k}^k {\rm sgn}(k)(-1)^{k+j}q^{k(3k+1)/2 +(m-1)(k^2+k)- j^2},  \ \ \text{$n$ even}.
\end{dcases}
\end{aligned}
\end{equation}
Note that for $m > 1$ the multisum on the left-hand side above does not truncate at odd roots of unity without the upper bound $n-1$.   

Next we let $q=q^2$ and apply Lemma \ref{Baileylemmabcinfty} $m-1$ times to the Bailey pair in \eqref{indefpair4}.   The result is the Bailey pair relative to $q^2$,
\begin{equation*}
\alpha_n = \frac{q^{2mn^2+(2m-1)n}(1-q^{2n+1})}{1-q}\sum_{j=-n}^n (-1)^jq^{-j^2}
\end{equation*}
and
\begin{equation*}
\beta_n = \beta_{n_m} = \frac{1}{(q^2;q^2)_{n_m}}\sum_{n_m \geq \cdots \geq n_1 \geq 0} \frac{(q^2;q^2)_{n_1}}{(-q^2)_{2n_1}} \prod_{i=1}^{m-1} q^{2n_i^2+2n_i} \begin{bmatrix} n_{i+1} \\ n_i \end{bmatrix}_{q^2}.
\end{equation*}
Inserting this pair into Lemma \ref{keylemma2}, we obtain
\begin{equation*}
\begin{aligned}
&\sum_{n-1 \geq n_m \geq \cdots \geq n_1 \geq 0} \frac{(-q^{2n_1+2})_{2n_m-2n_1}(q^{2-2n};q^2)_{n_m}(-1)^{n_m}(q^2;q^2)_{n_1}}{(q^2;q^2)_{n_m}} \\
& \qquad \qquad \qquad\times q^{n(2n_m+1)-(n_m+1)^2} \prod_{i=1}^{m-1} q^{2n_i^2+2n_i} \begin{bmatrix} n_{i+1} \\ n_i \end{bmatrix}_{q^2} \\
&= \frac{(-q)_{n-1}}{(q)_{n}}\sum_{k=0}^{n-1} \sum_{j=-k}^k \frac{(q^{1-n})_k}{(q^{1+n})_k}(-1)^{k+j}q^{nk +k(3k+1)/2 +(2m-2)(k^2+k)-j^2}(1-q^{2k+1}).
\end{aligned}
\end{equation*}

Now we wish to let $q$ be a primitive $n$th root of unity.    On the left-hand side, if $n$ is odd, then $q^2$ is also a primitive $n$th root of unity and so
\begin{equation*}
\frac{(q^{2-2n};q^2)_{n_m}}{(q^2;q^2)_{n_m}} = 1. 
\end{equation*}
If $n$ is even, then this term takes the form $0/0$ when $n_m \geq n/2$.    But then either $n_1 \geq n/2$, in which case $(q^2;q^2)_{n_1} = 0$, or $n_1 < n/2$, in which case $(-q^{2n_1+2})_{2n_m - 2n_1} = 0$.    So this is never an issue (and in fact, the sum actually truncates at $n/2-1$).    The right-hand side evaluates as usual, and we obtain
\begin{equation} \label{theorem3proofbigeq2}
\begin{aligned}
&\sum_{n-1 \geq n_m \geq \cdots \geq n_1 \geq 0} (-q^{2n_1+2})_{2n_m-2n_1}(-1)^{n_m}q^{-(n_m+1)^2}{(q^2;q^2)_{n_1}} \prod_{i=1}^{m-1} q^{2n_i^2+2n_i} \begin{bmatrix} n_{i+1} \\ n_i \end{bmatrix}_{q^2} \\
&= \begin{dcases}
\frac{-1}{n^2} \sum_{k=-n}^{n-1} \sum_{j=-k}^k {\rm sgn}(k) (k(3k+1)/2 +(2m-2)(k^2+k) - j^2)
\\ \qquad \qquad \qquad \times (-1)^{k+j}q^{k(3k+1)/2 + (2m-2)(k^2+k)- j^2}, \qquad \text{$n$ odd}, \\
\frac{1}{4} \sum_{k=-n}^{n-1} \sum_{j=-k}^k {\rm sgn}(k)(-1)^{k+j}q^{k(3k+1)/2 +(2m-2)(k^2+k)- j^2},  \ \ \text{$n$ even}.
\end{dcases}
\end{aligned}
\end{equation}
Comparing this with \eqref{theorem3proofbigeq1} at $m \mapsto 2m-1$ gives the result.   
\end{proof}

\begin{proof}[Proof of Theorem \ref{theorem4}]
Apply Lemma \ref{Baileylemmabcinfty} $m-1$ times with $q=q^2$ to the Bailey pair in \eqref{eq:BP1.4}. We obtain the following Bailey pair relative to $q^2$:
\begin{equation*}
\alpha_n = \frac{q^{2n^2}(1-q^{4n+2})}{1-q^2}\sum_{j=-n}^n (-1)^jq^{-j^2}
\end{equation*}
and
\begin{equation*}
\beta_n = \beta_{n_m} = \frac{1}{(q^2;q^2)_{n_m}}\sum_{n_m \geq \cdots \geq n_1 \geq 0} \frac{(q;q^2)_{n_1}}{(-q)_{2n_1}} \prod_{i=1}^{m-1} q^{2n_i^2+2n_i} \begin{bmatrix} n_{i+1} \\ n_i \end{bmatrix}_{q^2}.
\end{equation*}
Inserting this pair into Lemma \ref{keylemma3}, and taking $q$ to be a primitive $n$th odd root of unity leads to
\begin{equation*}
\begin{aligned}
\sum_{n-1 \geq n_m \geq \cdots \geq n_1 \geq 0} &(-q^{2n_1+1})_{2n_m-2n_1}(-1)^{n_m}q^{-n_m^2}(q;q^2)_{n_1} \prod_{i=1}^{m-1} q^{2n_i^2+2n_i} \begin{bmatrix} n_{i+1} \\ n_i \end{bmatrix}_{q^2}\\
&=\frac{(-q)_{n-1}}{(q)_{n}}\sum_{k=0}^{n-1} \sum_{j=-k}^k (-1)^{k+j}q^{k(3k+1)/2 +(2m-2)(k^2+k)-j^2}(1-q^{2k+1}).
\end{aligned}
\end{equation*}
Arguing as usual gives
\begin{equation*} \label{theorem4proofbigeq}
\begin{aligned}
\sum_{n-1 \geq n_m \geq \cdots \geq n_1 \geq 0} &(-q^{2n_1+1})_{2n_m-2n_1}(-1)^{n_m}q^{-n_m^2}(q;q^2)_{n_1} \prod_{i=1}^{m-1} q^{2n_i^2+2n_i} \begin{bmatrix} n_{i+1} \\ n_i \end{bmatrix}_{q^2}\\
&=\frac{-1}{n^2} \sum_{k=-n}^{n-1} \sum_{j=-k}^k {\rm sgn}(k) (k(3k+1)/2 +(2m-2)(k^2+k) - j^2)
\\ &\qquad \qquad \qquad \times (-1)^{k+j}q^{k(3k+1)/2 + (2m-2)(k^2+k)- j^2}.
\end{aligned}
\end{equation*}
Comparing this with \eqref{theorem3proofbigeq2} gives the result.   
\end{proof}

\section*{Declarations}

The authors declare that there are no conflicts of interest.

The authors were supported by the SNSF Eccellenza grant PCEFP2 202784 and the ANR project Combin\'e (ANR-19-CE48-0011).

\end{document}